\documentclass[11pt,a4paper]{article}
\usepackage{amssymb,amsmath}
\usepackage{amsthm}
\usepackage{bbm}
\usepackage{stmaryrd}
\usepackage{hyperref}
\usepackage[usenames,dvipsnames]{xcolor}
\usepackage{color}

\input xy
\xyoption{all} \tolerance=500

\font\black=cmbx10 \font\sblack=cmbx7 \font\ssblack=cmbx5 \font\blackital=cmmib10  \skewchar\blackital='177
\font\sblackital=cmmib7 \skewchar\sblackital='177 \font\ssblackital=cmmib5 \skewchar\ssblackital='177
\font\sanss=cmss10 \font\ssanss=cmss8 
\font\sssanss=cmss8 scaled 600 \font\blackboard=msbm10 \font\sblackboard=msbm7 \font\ssblackboard=msbm5
\font\caligr=eusm10 \font\scaligr=eusm7 \font\sscaligr=eusm5  \font\fraktur=eufm10
\font\sfraktur=eufm7 \font\ssfraktur=eufm5

\font\bsymb=cmsy10 scaled\magstep2
\def\all#1{\setbox0=\hbox{\lower1.5pt\hbox{\bsymb
       \char"38}}\setbox1=\hbox{$_{#1}$} \box0\lower2pt\box1\;}
\def\exi#1{\setbox0=\hbox{\lower1.5pt\hbox{\bsymb \char"39}}
       \setbox1=\hbox{$_{#1}$} \box0\lower2pt\box1\;}

\def\tx#1{{\fam0\relax#1}}

\newfam\bifam
\textfont\bifam=\blackital \scriptfont\bifam=\sblackital \scriptscriptfont\bifam=\ssblackital

\newfam\blfam
\textfont\blfam=\black \scriptfont\blfam=\sblack \scriptscriptfont\blfam=\ssblack

\newfam\bbfam
\textfont\bbfam=\blackboard \scriptfont\bbfam=\sblackboard \scriptscriptfont\bbfam=\ssblackboard

\newfam\ssfam
\textfont\ssfam=\sanss \scriptfont\ssfam=\ssanss \scriptscriptfont\ssfam=\sssanss
\def\sss#1{{\fam\ssfam\relax#1}}

\newfam\clfam
\textfont\clfam=\caligr \scriptfont\clfam=\scaligr \scriptscriptfont\clfam=\sscaligr

\newfam\frfam
\textfont\frfam=\fraktur \scriptfont\frfam=\sfraktur \scriptscriptfont\frfam=\ssfraktur

\def\hpb#1{\setbox0=\hbox{${#1}$}
    \copy0 \kern-\wd0 \kern.2pt \box0}
\def\vpb#1{\setbox0=\hbox{${#1}$}
    \copy0 \kern-\wd0 \raise.08pt \box0}

\def\pmb#1{\setbox0\hbox{${#1}$} \copy0 \kern-\wd0 \kern.2pt \box0}
\def\pmbb#1{\setbox0\hbox{${#1}$} \copy0 \kern-\wd0
      \kern.2pt \copy0 \kern-\wd0 \kern.2pt \box0}
\def\pmbbb#1{\setbox0\hbox{${#1}$} \copy0 \kern-\wd0
      \kern.2pt \copy0 \kern-\wd0 \kern.2pt
    \copy0 \kern-\wd0 \kern.2pt \box0}
\def\pmxb#1{\setbox0\hbox{${#1}$} \copy0 \kern-\wd0
      \kern.2pt \copy0 \kern-\wd0 \kern.2pt
      \copy0 \kern-\wd0 \kern.2pt \copy0 \kern-\wd0 \kern.2pt \box0}
\def\pmxbb#1{\setbox0\hbox{${#1}$} \copy0 \kern-\wd0 \kern.2pt
      \copy0 \kern-\wd0 \kern.2pt
      \copy0 \kern-\wd0 \kern.2pt \copy0 \kern-\wd0 \kern.2pt
      \copy0 \kern-\wd0 \kern.2pt \box0}

\mathchardef\za="710B  
\mathchardef\zb="710C  
\mathchardef\zg="710D  
\mathchardef\zd="710E  
\mathchardef\zve="710F 
\mathchardef\zz="7110  
\mathchardef\zh="7111  
\mathchardef\zvy="7112 
\mathchardef\zi="7113  
\mathchardef\zk="7114  
\mathchardef\zl="7115  
\mathchardef\zm="7116  
\mathchardef\zn="7117  
\mathchardef\zx="7118  
\mathchardef\zp="7119  
\mathchardef\zr="711A  
\mathchardef\zs="711B  
\mathchardef\zt="711C  
\mathchardef\zu="711D  
\mathchardef\zvf="711E 
\mathchardef\zq="711F  
\mathchardef\zc="7120  
\mathchardef\zw="7121  
\mathchardef\ze="7122  
\mathchardef\zy="7123  
\mathchardef\zf="7124  
\mathchardef\zvr="7125 
\mathchardef\zvs="7126 
\mathchardef\zf="7127  
\mathchardef\zG="7000  
\mathchardef\zD="7001  
\mathchardef\zY="7002  
\mathchardef\zL="7003  
\mathchardef\zX="7004  
\mathchardef\zP="7005  
\mathchardef\zS="7006  
\mathchardef\zU="7007  
\mathchardef\zF="7008  
\mathchardef\zW="700A  
\mathchardef\zC="7009  

\newcommand{\be}{\begin{equation}}
\newcommand{\ee}{\end{equation}}

\newcommand{\bea}{\begin{eqnarray}}
\newcommand{\eea}{\end{eqnarray}}
\newcommand{\beas}{\begin{eqnarray*}}
\newcommand{\eeas}{\end{eqnarray*}}

\def\*{{\textstyle *}}
\newcommand{\R}{{\mathbb R}}

\newcommand{\s}{{\textstyle *}}

\newcommand{\ti}{\times}

\def\sT{{\sss T}}

\def\xi{\tx{i}}

\newdir{ (}{{}*!/-5pt/@^{(}}

\def\s*{{\scriptstyle *}}

\def\rel{-\!\!-\triangleright}
\newcommand{\bfr}{\begin{frame}}
\newcommand{\efr}{\end{frame} }

\def\rel{{-\!\!\!-\!\!\rhd}}
\newdir{|>}{%
!/4.5pt/@{|}*:(1,-.2)@^{>}*:(1,+.2)@_{>}}
\newdir{ (}{{}*!/-5pt/@^{(}}
\newcommand{\defequal}{\stackrel{\mbox {\tiny {def}}}{=}}

\newcommand{\inverse}{^{-1}}
\newcommand{\calf}{{\cal F}}

\newtheorem{theorem}{Theorem}[section]
\newtheorem{proposition}[theorem]{Proposition}

\theoremstyle{definition}
\newtheorem{definition}[theorem]{Definition}
\newtheorem{remarkth}[theorem]{Remark}
\newtheorem{example}[theorem]{Example}

\begin{document}
\bibliographystyle{plain}

\author{\\
        {\Large Janusz Grabowski}}

\date{}
\title{\bf An introduction to loopoids\thanks{Research funded by the  Polish National Science Centre grant under the contract number DEC-2012/06/A/ST1/00256.}}
\maketitle
\begin{abstract}
We discuss a concept of \emph{loopoid} as a non-associative generalization of Brandt groupoid. We introduce and study also an interesting class of more general objects which we call \emph{semiloopoids}. A differential version of loopoids is intended as a framework for Lagrangian discrete mechanics.
\end{abstract}



\maketitle


\section{Introduction}
Compared to the theory of groups, the theory of quasi-groups is considerably
older, dating back at least to Euler's work on orthogonal Latin squares. But later, the theory of
quasi-groups was eclipsed by the phenomenal development of the theory of groups and Lie groups.
With the initial completion of the classification of the finite simple groups, however, attention is once again becoming more
evenly divided between the two theories. Also the theory of smooth quasi-groups and loops started to find interesting applications in geometry and physics.

We refer to the books \cite{Br,Bel, Pfl, Sab} and the survey articles \cite{Br1,Sab1,Smi} if terms and concepts from non-associative algebra, especially loops, are concerned.
However, for completeness and reader's convenience, we recall basic definitions.

Let us recall that a \emph{quasi-group} is is an algebraic structure $<G,\cdot>$ with a binary operation (written usually as juxtaposition, $a\cdot b=ab$) such
that
$r_g: x \mapsto xg$ (the \emph{right translation}) and
$l_g: x \mapsto gx$ (the \emph{left translation})
are permutations of $G$, equivalently, in which the equations
$ya = b$ and $ax = b$ are soluble uniquely for $x$ and $y$ respectively. If we assume only that left (resp., right) translations are permutations, we speak about a \emph{left quasi-group} (resp., \emph{right quasi-group}. A \emph{left loop} is defined to be a left quasi-group with a right identity $e$, i.e. $xe=x$, while a \emph{right loop} is a right quasi-group with a left identity, $ex=x$. A \emph{loop} is a quasi-group with a two-sided identity element,
$e$, $e x=x e=x$. A loop $<G,\cdot , e>$ with identity $e$ is called an \emph{inverse
loop} if to each element $a$ in G there corresponds an element $a^{-1}$ in $G$ such that
$$a^{-1}(a b) = (b a) a^{-1} =b$$ for all $b\in G$.
It can be easily shown that in an inverse loop $<G, \cdot, {}^{-1}, e>$ we have,
for all $a, b \in G$,
$$aa^{-1} = a^{-1}a = e,\quad (a^{-1})^{-1} = a,\quad \text{and}\quad  (ab)^{-1} = b^{-1} a^{-1}\,.$$
Loops, or more generally left loops, appear naturally as algebraic structures
on \emph{transversals} or \emph{sections} of a subgroup of a group. This observation, going
back to R. Baer \cite{Bae} (cf. also \cite{Fo,KW}), lies at the heart of much current research on loops, also in differential geometry and analysis.

\begin{example}\label{tr} Let $G$ be a group with the unit $e$, $H$ be its subgroup, and $S\subset G$ be a left transversal to $H$ in $G$, i.e. $S$ contains exactly one point from each coset $gH$ in $G/H$. This means that any element $g\in G$ has a unique decomposition $g=sh$, where $s\in S$ and $h\in H$. This produces an identification $G=S\times H$ of sets. Let $p_S:G\to S$ be the projection on $S$ determined by this identification. If we assume that $e\in S$, then $S$ with the multiplication $$s\circ s'=p_S(ss')$$ and $e$ as a right unit is a left loop.

Indeed, as $e\circ s=s\circ e=p_S(s)=s$, $e$ is the unit for this multiplication. For $a,b\in S$, there is $h\in H$ such that $p_S(a^{-1}b)=a^{-1}bh$. Hence,
$$p_S(ap_S(a^{-1}b))=p_S(aa^{-1}bh)=p_S(bh)=b\,,$$
that shows that $p_S(a^{-1}b)$ is a solution of the equation $a\circ x=b$. If $c,c'$ are two such solutions, then $p_S(ac)=p_S(ac')$, so there is $h\in H$ such that $ac=ac'h$, so $c=c'h$ and $c=c'$, since $S$ is transversal to $H$.
\end{example}

In this paper, we would like to propose a concepts of \emph{loopoid}, defined as a nonassociative generalization of a groupoid. Note that here and throughout the paper, by \emph{groupoid} we understand a \emph{Brandt groupoid}, i.e. a small category in which every morphism is an isomorphism, and not an object called in algebra also a \emph{magma}.
These are loops which can be considered as nonassociative generalizations of groups. In the case of genuine groupoids, however, the situation is more complicated, because the multiplication is only partially defined, so the axioms of a loop  must be reformulated.

A convenient way is to think about groupoids as being defined exactly like groups but with the difference that all objects/maps in the definition are relations, like it has been done by Zakrzewski \cite{Zak}. In particular, the unity is a relation $\ze:\{ e\}\rel\ G$, associating to a point $e$ a subset $M=\ze(e)\subset G$, the set of units. Using this idea, we define semiloopoids, as well as more specific objects which we will call \emph{loopoids}.

We want to stress that our motivation comes from discrete mechanics, where Lie groupoids have been recently used for a geometric formulation of the Lagrangian formalism \cite{FZ,IMMM,IMMP,MMM,MMS,Stern,weinstein}. Infinitesimal parts of Lie groupoids are Lie algebroids and the corresponding `Lie theory' is well established (cf. \cite{Ma}).
We believe that this can be extended to a differential version of the concept of (semi)loopoid, a \emph{differential (semi)loopoid}.  

As the infinitesimal version of associativity is the Jacobi identity, the corresponding `brackets' will not satisfy the latter.
Note that in the literature there are already natural various generalizations of Lie algebroids, e.g. \emph{skew algebroids}, \emph{almost Lie algebroids}, or \emph{Dirac algebroids} \cite{GG,GG1,GJ,GU1,GU2}, where no Jacobi identity is assumed. For instance, the skew algebroid formalism is very useful in describing the geometry of nonholonomic systems \cite{GLMM}. We believe that we can obtain skew and/or almost Lie algebroids as infinitesimal parts of differential loopoids and that standard geometric constructions of the tangent and cotangent groupoid for a given Lie groupoid can be extended to this category. We postpone, however, these questions to a separate paper.

Note finally, that after writing the first version of these notes, we learned that the term \emph{loopoid} has appeared already in a paper by Kinyon \cite{Kin} in a similar context. The motivating example, however, built as an object `integrating' the Courant bracket on $\sT M\oplus_M\sT^*M$, uses the group of diffeomorphisms of the manifold $M$ as integrating the Lie algebra of vector fields on $M$, not the pair groupoid $M\ti M$ as `integrating' the Lie algebroid $\sT M$.

\section{Groupoids}

\begin{definition}
A \emph{groupoid} over a set $M$ is a set $G $ equipped
with source and target mappings $\alpha,\beta:G\to M$,  a
multiplication map $m$ from $G_{2}\defequal \{(g,h)\in
G \times G |\ \beta (g)=\alpha (h)\}$ to $G $,  an  injective
units mapping $\epsilon :M\rightarrow G $,  and an
inversion mapping $\iota :G \rightarrow G$,
satisfying the following properties (where we write $gh$ for $m(g,h)$
and $g\inverse$ for $\iota (g)$):
\begin{itemize}

\item {(associativity)}  $g(hk)=(gh)k$ in the sense that, if
one side of the equation is defined, so is the other, and then they
are equal;

\item {(identities)} $\epsilon (\alpha (g))g=g=g\epsilon (\beta (g))$;

\item {(inverses)} $gg\inverse =\epsilon(\alpha (g))$ and $g\inverse
g=\epsilon (\beta (g))$.
\end{itemize}
\end{definition}
\vskip-.1cm  The elements of $G _{2}$ are sometimes referred to as \emph{composable} (or \emph{admissible}) pairs. A groupoid $G$ over a set $M$ will be denoted
$G \rightrightarrows M$. Note that the full information about the groupoid is contained in the
\emph{multiplication relation} which is a subset $G_3\subset G\ti G\ti G$,
\be\label{mr}
G_3=\left\{(x,y,z)\in G\ti G\ti G\, |\ (x,y)\in G_2\ \text{and}\ z=xy\right\}\,.
\ee

\begin{example}\label{e1} ({\bf pair groupoid}) Let $M$ be a set and $G=M\ti M$ and
$$\za(u,v)=u\,, \quad\zb(u,v)=v\,.$$
Then, $M\ti M$ is a groupoid over $M$ with the source and target maps $\za,\zb$, units mapping $\ze(u)=(u,u)$, and the partial composition by $(u,v)(v,z)=(u,z)$. In other words,
$$G_3=\left\{(u,v,v,z,u,z)\in G\ti G\ti G\, |\ u,v,z\in M\right\}\,.$$
\end{example}

\begin{remarkth}
We can regard $M$, \emph{via} the embedding $\ze$, as a subset in $G$, and thus $\ze$ as the identity, that simplifies the picture, since $\za,\zb$ become just projections in $G$. Indeed, in view of associativity, 
$$(\ze(\za(g))^2g=(\ze(\za(g))(((\ze(\za(g))g)=\ze(\za(g))g\,,
$$
so that $(\ze(\za(g))^2=\ze(\za(g))=\ze(\za(\ze(\za(g))))\ze(\za(g))$
and, consequently
$$\ze(\za(\ze(\za(g))))=\ze(\za(g))\,,$$
i.e. $\ze\circ\za$ is a projection, $(\ze\circ\za)^2=\ze\circ\za$. Similarly, $\ze\circ\zb$ is a projection.
We will use this convention in the sequel.
\end{remarkth}

\section{Semiloopoids}
Following the Zakrzewski's idea of obtaining the definition of a groupoid by replacing the objects in the definition of a group by relations, we propose the following.
\begin{definition}
 A \emph{semiloopoid over a set $M$} is a structure consisting of a set $G$ together with projections $\za,\zb:G\to M$ onto a subset $M\subset G$ (\emph{set of units}) and a multiplication relation $G_3\subset G\ti G\ti G$ such that, for each $g\in G$,
\be\label{units} (\za(g),g,g)\in G_3\quad\text{and}\quad (g,\zb(g),g)\in G_3\,,
 \ee
and the relations $l_g,r_g\subset G\ti G$ defined by
 \bea\label{multl} (h_1,h_2)\in l_g\ &\Leftrightarrow& (g,h_1,h_2)\in G_3\,,\\
 (h_1,h_2)\in r_g\ &\Leftrightarrow& (h_1,g,h_2)\in G_3\,.\label{multr}
 \eea
 are injective.
If we forget condition (\ref{multr}) (resp., (\ref{multl})), then we speak about a \emph{left} (resp. \emph{right}) \emph{semiloopoid}.
 
 A \emph{semiloopoid morphism} between semiloopoids $G,H$ over $M,N$, respectively, is is a pair of maps $(\Phi, \phi)$, where $\Phi:G\to H$ and $\phi:M\to N$, satisfying
 $$\Phi_{|M}=\phi\,,\quad \za_H\circ\Phi=\phi\circ\za_G\,,\quad \zb_H\circ\Phi=\phi\circ\zb_G\,,
 $$
 and such that $(\Phi,\Phi,\Phi):G\ti G\ti G\to H\ti H\ti H$ maps $G_3$ into $H_3$.
 The last condition means:
 \be\label{mor}\Phi(gh)=\Phi(g)\Phi(h)\,,
 \ee
provided $(g,h)\in G_2$.
 \end{definition}

 Denote the range of the projection of $G_3$ onto the first two factors $G\ti G$ with $G_2$. Condition
 (\ref{multl}) (or (\ref{multr})) implies that $G_3$ is actually the graph of a map
 $m:G_2\to G$, so we can write $z=m(g,h)$, or simply $z=gh$, instead of $(g,h,z)\in G_3$.
 In particular, in the above notation, $\za(g)g=g$ and $g\zb(g)=g$.

 Consequently, we will write $l_gh=z$ and $r_gh=z$ instead of $(h,z)\in l_g$ and
 $(h,z)\in r_g$, respectively. We can therefore view $l_g$ and $r_g$ as bijections defined on their domains, $D_g^l$ and $D_g^r$ onto their ranges, $R_g^l$ and $R_g^r$, respectively.
The definition of a semiloopoid can be therefore reformulated in a more instructive way as follows.

\begin{definition} {\bf (alternative)}
A \emph{semiloopoid over a set $M$} is a structure consisting of a set $G$ including $M$ and equipped with
\begin{itemize}

\item a \emph{partial multiplication} $m: G\ti G\supset G_{2} \to G$, $m(g,h)=gh$,
such that, for all $g\in G$,
\be\label{l}l_g:D_g^l\to R_g^l\,,\ l_gh=gh\,,
 \ee
 is a bijection from $D_g^l=\{ h\in G\, |\, (g,h)\in G_2\}$ onto $R_g^l=\{ gh\, |(g,h)\in G_2\}$, and
 \be\label{r}r_g:D_g^r\to R_g^r\,,\ r_gh=hg\,,
  \ee
  is a bijection from $D_g^r=\{ h\in G\, |\, (h,g)\in G_2\}$ onto $R_g^r=\{ hg\, |(h,g)\in G_2\}$;

\item a pair of projections $\alpha,\beta: G \to M$ such that, for all $g\in G$,
\be\label{e}\alpha(g)g=g\,,\ g\beta(g)=g\,.
\ee
\end{itemize}
\end{definition}
\begin{remarkth}
Note that in a semiloopoid all the structural maps, $\za,\zb,m$, are determined by just $G_3$ (cf. (\ref{units})). 

\end{remarkth}

\begin{example}\label{e2} {\bf (trivial semiloopoid over $M$)} On a set $G$ including $M$ let us choose
projections $\za,\zb:G\to M$ and put
$$G_2=\{ (\za(g),g)\,|\, g\in G\}\cup\{ (g,\zb(g))\,|\, g\in G\}\,.$$
The map $m:G_2\to G$, given by $m(\za(g),g)=g=m(g,\zb(g))$, establishes on $G$ a structure of a semiloopoid  over $M$.
\end{example}
\begin{example}\label{ex}
We can make the above example more complicated, choosing $g_0\in G\setminus M$, a subset $A\subset G$, $A\cap M=\{\zb(g_0)\}$ and an injective map $l_0:A\to G\setminus M$ such that $l(\zb(g_0))=g_0$. Then, we obtain a semiloopoid by putting
$$G_2=\{ (\za(g),g)\,|\, g\in G\}\cup\{ (g,\zb(g))\,|\, g\in G\}\cup\{ (g_0,h)\,|\, h\in A\}\,,$$
and the partial multiplication $m:G_2\to G$ which, besides the unity property (\ref{e}), satisfies
$$m(g_0,h)=l_0(h)\,,\ \text{for}\ h\in B\,.
$$
\end{example}
\begin{definition} A semiloopoid will be called a \emph{left inverse semiloopoid} if there is a \emph{left inversion map} $\zi_l: G \to G$ such that for each $(g,h)\in G_2$ also $(\zi_l(g),gh)\in G_2$ and $\zi_l(g)(gh)=h$. A \emph{right inverse semiloopoid} can be defined analogously.

A semiloopoid will be called an \emph{inverse semiloopoid} if there is an \emph{inversion map} $\zi: G \to G$, to be denoted simply by $\zi(g)=g^{-1}$, such that, for each $(g,h),(u,g)\in G_2$, also $(g^{-1},gh),(ug,g^{-1})\in G_2$ and
$$g^{-1}(gh)=h\,,\quad (ug)g^{-1}=u\,.
$$
\end{definition}
\begin{proposition} In any inverse semiloopoid the following hold true:
\be\label{il} g^{-1}g=\zb(g)=\za(g^{-1})\,,\quad gg^{-1}=\za(g)=\zb(g^{-1})\,,\quad \left(g^{-1}\right)^{-1}=g\,,\quad (gh)^{-1}=h^{-1}g^{-1}\,.
\ee
The latter condition means that one side of the equality makes sens if and only if the other makes sense (the elements are composable) and they are equal.
\end{proposition}
\begin{proof}
By definition of the inverse, $g^{-1}g=g^{-1}(g\zb(g))=\zb(g)$ and, similarly, $gg^{-1}=\za(g)$. Now,
$$\left(g^{-1}\right)^{-1}\zb(g)=\left(g^{-1}\right)^{-1}(g^{-1}g)=g\,.$$
But also $g\zb(g)=g$, thus $g=\left(g^{-1}\right)^{-1}$ and
$$\zb(g)g^{-1}g=g^{-1}\left(g^{-1}\right)^{-1}=\za(g^{-1})\,.$$
Consequently, $\zb(g^{-1})=\za(g)$ and, finally,
$$(gh)^{-1}=h^{-1}(h(gh)^{-1})=h^{-1}\left((g^{-1}(gh))(gh)^{-1}\right)=h^{-1}g^{-1}\,.
$$
\end{proof}
\begin{example}
In example \ref{e2} we can use an involutive bijection $\zi$ of $G$, intertwining $\za$ and $\zb$, for extending the partial multiplication so that we will obtain an inverse semiloopoid.
\end{example}
\section{Tranversals}
In the context of semiloopoids, we want to define \emph{transversals}, similar but more general than these described in Example \ref{tr}, to produce new examples of semiloopoids.
\begin{definition}\label{trans}
Let $G$ be a semiloopoid. A pair $(T,\pi)$, where $T\subset G$ and $\pi:G\to T$ is a projection is called a \emph{transversal} in $G$ if
\begin{itemize}
\item \be\label{0}\za(T),\zb(T)\subset T\,,\ee
\item  for $t\in T$, the relations $l_t^T,r_t^T\subset T\ti G$, reducing the relations $l_t,r_t$ to $T\ti G$, i.e.
 \bea\label{multln} (t',g)\in l_t^T\ &\Leftrightarrow& (t,t',g)\in G_3\,,\\
 (t',g)\in r_t^T\ &\Leftrightarrow& (t',t,g)\in G_3\,,\label{multrn}
 \eea
 are \emph{transversal to $\pi$} in the sense that the composition relations $\pi\circ l_t,\pi\circ r_t\subset T\ti T$ are injective. In other words, for $t,t',u,u'\in T$, if $tt'$ and $uu'$ are different elements in the same fiber of $\pi$, then $t\ne t'$ and $u\ne u'$.
\end{itemize}
 If we forget condition (\ref{multrn}) (resp., (\ref{multln})), then we speak about a \emph{left}
(resp. \emph{right}) \emph{transversal}.
\end{definition}
\begin{example}
The transversal described in Example \ref{tr} is an example of a left transversal in $G$ in the above sense. In this case, $T=S$ and $\zp=p_S$. The first condition of Definition \ref{trans} means that $e\in S$. Condition (\ref{multln}), in turn, means
that, for $s,s',s''\in S$, if $p_S(ss')=p_S(ss'')$, then $s'=s''$. It is clearly satisfied, since $p_S(ss')=p_S(ss'')$ implies that there is $h,\in H$ such that $ss''=ss'h$, thus $s''=s'h$.
From the uniqueness of the decomposition $G=SH$ we get $s''=s'$.
\end{example}
\begin{proposition}
If $(T,\pi)$ is a (left, right) transversal in a semiloopoid $G$ with the source and target maps $\za,\zb:G\to M$, then $T$ is a (left, right) semiloopoid itself with the source and target maps
$$\za_{|T},\zb_{|T}:T\to M_T:=\za(T)=\zb(T)\subset T$$
and the partial multiplication $m_T:T_2:=G_2\cap(T\ti T)\to T$
$$t\bullet t'=\pi(tt')\,.
$$
\end{proposition}
\begin{proof} Since $\za(\zb(T))=\zb(T)\subset T$, we have trivially $\za(T)=\zb(T)$ as the set of units in $T$.
The injectivity of the left and/or right translations follows from the corresponding transversality properties.
\end{proof}
\section{Loopoids}
As shows example \ref{ex}, the maps $\za,\zb$ in can be rather pathological, if their general properties are concerned.
Let us assume now, that a semiloopoid $G$ over $M$, with a partial multiplication $m$ and projections $\za,\zb:G\to M$, satisfies a very weak associativity condition, hereafter called \emph{unities associativity}:
\be\label{ua} (xy)z=x(yz)\ \text{if one of}\ x,y,z\ \text{is a unit}\ (\text{i.e. belongs to}\ M)\,.
\ee
The above condition has to be understood as follows: if one side of equation (\ref{ua}) makes sense, the other makes sense and we have equality.
The following proposition shows that the condition of unities associativity for a semiloopoid over $M$ is rather strong and implies that the \emph{anchor map} $(\za,\zb):G\to M\ti M$ has nice properties, similar to these for groupoids.
\begin{proposition}
A semiloopoid $G$ over $M$ satisfyies the unities associativity condition if and only if
\be\label{ag1}
G_{2}= \{(g,h)\in
G \times G \,|\ \beta (g)=\alpha (h)\}
\ee
and
\be\label{ag2}
(\za,\zb):G\to M\ti M
\ee
is a semiloopoid morphism into the pair groupoid $M\ti M$, i.e.
\be\label{agc}
\za(gh)=\za(g)\ \text{and}\ \zb(gh)=\zb(h)\,.
\ee
\end{proposition}
\begin{proof}
If $(g,h)\in G_2$, then according to (\ref{ua}),
$$g(\za(h)h)=gh=(g\zb(g))h=g(\zb(g)h)\,,$$
so, in view of injectivity of $l_g$, we have $h=\za(h)h=\zb(g)h$ and, consequently,
$\zb(g)=\za(h)$. Similarly,
$$\za(gh)(gh)=gh=(\za(g)g)h=\za(g)(gh)\,,$$
so $\za(gh)=\za(g)$. Analogously we can prove $\zb(gh)=\zb(h)$.

Conversely, let $e\in M$ be such that $e(gh)$ makes sense. Then,
$$e=\zb(e)=\za(gh)=\za(g)\,.$$
Hence
\be\label{a}e(gh)=\za(gh)(gh)=gh=(\za(g)g)h=(eg)h\,.
\ee
If this is $(eg)h$ that makes sense, then
$$e=\zb(e)=\za(g)=\za(gh)$$ and we have (\ref{a}) again.
Similarly we prove $(ge)h=g(eh)$ and $(gh)e=g(he)$.
\end{proof}
The inverse images of points under the source and target maps we call
$\alpha $- and \emph{$\beta $-fibres}.  The fibres through a point $g$,
will be denoted by $\calf ^{\alpha }(g)$ and $\calf ^{\beta }(g)$, respectively.
The unities associativity assumption implies that each element $g$ of $G
$ determines the \emph{left and right translation maps}
\be\label{aga}{l}_{g}:\calf ^{\alpha
}(\beta (g))\rightarrow \calf ^{\alpha }(\alpha  (g))\,, \quad
r_{g}:\calf ^{\beta }(\alpha  (g))\rightarrow \calf ^{\beta }(\beta
(g))\,,
\ee
which are injective.
\begin{definition}
A semiloopoid satisfying the unities associativity assumption  and such that the maps (\ref{aga}) are bijective will be called a \emph{loopoid}.
\end{definition}
\begin{remarkth}
In a loop, the multiplication is globally defined, so the unity associativity is always satisfied
by properties of the unity element. In this sense, loops are loopoids over one point.
\end{remarkth}
\begin{proposition}
Let $G$ be a loopoid over $M$ with the source and target maps $\za,\zb:G\to M$. Then, for each $u\in M$, the multiplication in $G$ induces on the set
$$G_u=\{ g\in G\,|\, \za(g)=\zb(g)=u\}$$ a loop structure.
\end{proposition}
\begin{proof}
$G_u$ is clearly closed with respect to the multiplication and this multiplication is globally defined, $G_u\subset G_2$. As $u$ is the only unit in $G_u$, we have only one $\za$/$\zb$-fiber
on which translations act as bijections, thus we deal with a loop.
\end{proof}
\noindent The loop $G_u$ above will be called the \emph{isotropy loop} of $u\in M$.
\begin{example}
Let $X$ be a loop with the unit $e$ and $N$ be a set. On $G=X\ti N\ti N$ we have an obvious structure of a loopoid as a product structure of the loop $X$ and the pair groupoid $N\ti N$ over $M= \{(e,s,s)\,|\, s\in N\}\subset G$. The anchor map is
$$(\za,\zb)(x,s,t)=(e,s,t)$$
and the partial multiplication reads
$$(x,s,t)\bullet(y,t,r)=(xy,s,r)\,.$$
If $X$ is an inverse loop, then $G$ is an inverse loopoid with the inverse $\zi(x,s,t)=(x^{-1},t,s)$. In this example $X$ is the \emph{isotropy loop} $G_u$ at each $u\in M$.
\end{example}

\noindent Note that we can consider slightly weaker objects tha loopoids.
\begin{definition}
A \emph{left loopoid} (resp., \emph{right loopoid}) is a semiloopoid satisfying $G_{2}= \{(g,h)\in G \times G \,|\ \beta (g)=\alpha (h)\}$ and 
such that, for each $g\in G$, the left translation ${l}_{g}$ is a bijection from $\calf ^{\alpha
}(\beta (g))$ onto $\calf ^{\alpha }(\alpha  (g))$ (resp., 
the right translation $r_{g}$ is a bijection from $\calf ^{\beta }(\alpha  (g))$ onto $\calf ^{\beta }(\beta
(g))$.
\end{definition}

\begin{example}
Let us finish with an interesting example of a left loopoid which is not a loopoid and can be regarded as a toy example of the intended \emph{left differential loopoid}. Consider the pair groupoid $\mathcal{G}=\R^2\ti\R^2$
with the standard source and target maps
$\za(u,v)=(u,u)$, $\zb(u,v)=(v,v)$
 and composition $(u,v)(v,z)=(u,z)$.
For a diffeomorphism $\zf:\R\to\R$ being an odd function, $\zf(-x)=-\zf(x)$, define a submanifold
\be\label{GG}G=\left\{\left((a_1,b_1),(a_2,b_2)\right)\in\mathcal{G}: a_1-a_2=\zf(b_1-b_2)\right\}\,.
\ee
It is a semiloopoid, with the source and target maps inherited from $\mathcal{G}$,
 and the partial multiplication
\be\label{pm1}\left((a_1,b_1),(a_2,b_2)\right)\bullet\left((a_2,b_2),(a_3,b_3)\right)=
\left((a_1,b_1),(a_1+\zf(b_3-b_1),b_3)\right)\,.
\ee
Indeed, for fixed $(a_1,b_1)\in\R^2$ the map $(a_3,b_3)\mapsto (a_1+\zf(b_3-b_1),b_3)$ is a diffeomorphism, so the left translations $l_{(a_1,b_1)}$ are smooth immersions. The right translations are smooth immersions trivially.
For $G$ we have $\za(gh)=\za(g)$, but generally, for non-linear $\zf$, we have $\zb(gh)\ne\zb(h)$, so $G$ is not a loopoid. 

It is interesting that $G$ has a left inverse $$\zi_l((a_1,b_1),(a_2,b_2))=((a_2,b_2),(a_1,b_1))\,.$$
Indeed, 
\beas&&\left((a_2,b_2),(a_1,b_1)\right)\bullet\left(\left((a_1,b_1),(a_2,b_2)\right)
\bullet\left((a_2,b_2),(a_3,b_3)\right)\right)\\
&=&\left((a_2,b_2),(a_1,b_1)\right)\bullet\left((a_1,b_1),(a_1+\zf(b_3-b_1),b_3)\right)\\
&=&\left((a_2,b_2),(a_2+\zf(b_3-b_2),b_3)\right)
=\left((a_2,b_2),(a_3,b_3)\right)\,.
\eeas
The last equality follows from the fact that $\left((a_2,b_2),(a_3,b_3)\right)\in G$, so by (\ref{GG}) 
$$a_2-a_3=\phi(b_2-b_3)=-\phi(b_3-b_2)\,.$$
All this implies that $G$ is a left (inverse) loopoid. Note that $G$ is actually a (Lie) groupoid if $\zf$ is linear and that $G$ is obtained as a transversal in the groupoid $\mathcal{G}$ with respect to a projection $\pi:\mathcal{G}\to G$ given by
$$\zp\left((a_1,b_1),(a_2,b_2)\right)=\left((a_1,b_1)(a_1+\zf(b_2-b_1),b_2)\right)\,.
$$
Indeed, the subset of units (the diagonal) belongs to $G$ (as $0=a_1-a_1=\zf(b_1-b_1)$), so $\za(G)=\zb(G)\subset G$. Moreover, if $\left((a_1,b_1),(a_2,b_2)\right),\left((a_2,b_2),(a_3,b_3)\right)\in G$, then
$$\zp\left(((a_1,b_1),(a_2,b_2))\cdot((a_2,b_2),(a_3,b_3))\right)=\zp\left((a_1,b_1),(a_3,b_3)\right)=
\left((a_1,b_1)(a_1+\zf(b_3-b_1),b_3)\right)\,,$$
so we recover the multiplication in $G$.
\end{example}


\small{
\noindent Janusz GRABOWSKI\\ Polish Academy of Sciences\\ Institute of
Mathematics\\ \'Sniadeckich 8, P.O. Box 21, 00-656 Warsaw,
Poland\\Email: jagrab@impan.pl \medskip
\end{document}

\end{document}
D